\documentclass{amsart}
\usepackage{amsmath, amsthm, amssymb}
\usepackage{relsize}
\theoremstyle{plain} 
\newtheorem{theorem}{Theorem}

\newtheorem{lemma}[theorem]{Lemma}
\def\us{\underset}

\begin{document}

\title{On the $L^1$ norm of an exponential sum involving the divisor function}
\author{D. A. Goldston and M. Pandey}
\thanks{$^{*}$ The first author was in residence at the Mathematical Sciences
Research Institute in Berkeley, California (supported by the National Science Foundation under Grant
No. DMS-1440140), during the Spring 2017 semester.}
%\address{D. A. Goldston, Department of Mathematics and Statistics,
 %San Jos\'{e} State University,
% San Jos\'{e}, California 95192-0103,
% United States of America}
\date{\today}
%\email{daniel.goldston@sjsu.edu}

\maketitle

\section{Introduction}
Let $\tau(n) = \sum_{d|n} 1$ be the divisor function, and 
\begin{align*}
S(\alpha) = \sum_{n\le x} \tau(n)e(n\alpha),  \qquad  e(\alpha) = e^{2\pi i \alpha}.
\end{align*}
In 2001 Brudern \cite{Brudern} considered the $L^1$ norm of $S(\alpha)$ and claimed to prove
\begin{equation}   \sqrt{x} \ll \int_0^1 |S(\alpha)| d\alpha \ll \sqrt{x}. \end{equation}
However there is a mistake in the proof given there which depends on a lemma which is false. 
In this note we prove the following result.
\begin{theorem}
We have 
\label{thm}
\begin{equation}
\sqrt{ x} \ll \int_0^1 |S(\alpha)|d\alpha\ll \sqrt{x}\log x.
\end{equation}
\end{theorem}
The upper bound here is obtained by following Brudern's proof with corrections. The lower bound is based on the method Vaughan introduced to study the $L^1$ norm for exponential sums over primes \cite{Va}, and also makes use of a more recent result of Pongsriiam and Vaughan \cite{PV} on the divisor sum in arithmetic progressions. We do not know whether the upper bound or the lower bound reflects the actual size of the $L^1$ norm here.

\section{Proof of the upper bound}
Let $u$ and $v$ always be positive integers. Following Br\"udern, we have
\[ \begin{split} 
S(\alpha) &=  \sum_{n\le x} \left( \sum_{uv=n}1\right) e(n\alpha) \\&=\sum_{uv\le x}e( uv\alpha) \\&=  2\sum_{u\le\sqrt{x}}\sum_{u < v\le x/u} e( uv\alpha) + \sum_{u\le\sqrt{x}} e( u^2\alpha)\\ &
:= 2T(\alpha) + V(\alpha) .
\end{split} \]
By  Cauchy-Schwarz and Parseval
\[ \int_0^1 |V(\alpha)|\, d\alpha \le \left( \int_0^1 |V(\alpha)|^2\, d\alpha \right)^{\frac12} = \sqrt{ \lfloor \sqrt{x}\rfloor} \le x^{\frac14} ,\]
and therefore by the triangle inequality
\[ \int_0^1 |S(\alpha)|\, d\alpha = 2 \int_0^1 |T(\alpha)|\, d\alpha + O(x^{\frac14}). \]
Thus to prove the upper bound in Theorem 1 we need to establish
\begin{equation} \label{integral}
\int_0^1 |T(\alpha)| d\alpha\ll\sqrt{x}\log x.
\end{equation}

We proceed as in the circle method. Clearly in \eqref{integral} we can replace the integration range $[0,1]$  by $[1/Q, 1+1/Q]$. By Dirichlet's theorem for any $\alpha \in [1/Q, 1+1/Q]$ we can  find a fraction $\frac{a}{q}$, $1\le q \le Q$,  $1\leq a \le q$, $(a,q)=1$, with $|\alpha - \frac{a}{q}| \le 1/(qQ)$. Thus the intervals $ [\frac{a}{q}- \frac{1}{qQ}, \frac{a}{q} +\frac{1}{qQ} ] $ cover the interval $[1/Q,1+1/Q]$.  Taking 
\begin{equation}  \label{Qequation}  2\sqrt{x} \le Q \ll \sqrt{x},  \end{equation}
we obtain
\[
\int_0^1 |T(\alpha)|d\alpha 
\le \sum_{q\le Q}\sum_{\substack{a = 1\\ (a,q)=1}}^q \int_{\frac{a}{q}-1/(2q\sqrt{x})}^{\frac{a}{q}+1/(2q\sqrt{x})} \left\vert T(\alpha)\right\vert d\alpha. 
\]
On each interval $ [\frac{a}{q}- \frac{1}{2q\sqrt{x}}, \frac{a}{q} +\frac{1}{2q\sqrt{x}} ] $ we decompose $T(\alpha)$
into
\[  T(\alpha) = F_q(\alpha) + G_q(\alpha) \]
where
\[  F_q(\alpha) = \sum_{\substack{u\le\sqrt{x}\\ q|u}}\sum_{u < v\le x/u} e( uv\alpha) \]
and 
\[  G_q(\alpha) = \sum_{\substack{u\le\sqrt{x}\\ q\nmid u}}\sum_{u < v\le x/u} e( uv\alpha), \]
and have
\[ \begin{split} \int_0^1 |T(\alpha)|d\alpha 
&\le  \sum_{q\le Q}\sum_{\substack{a = 1\\ (a,q)=1}}^q \int_{\frac{a}{q}-1/(2q\sqrt{x})}^{\frac{a}{q}+1/(2q\sqrt{x})} \left\vert F_q(\alpha)\right\vert d\alpha + \sum_{q\le Q}\sum_{\substack{a = 1\\ (a,q)=1}}^q \int_{\frac{a}{q}-1/(2q\sqrt{x})}^{\frac{a}{q}+1/(2q\sqrt{x})} \left\vert G_q(\alpha)\right\vert d\alpha\\ &
:=  I_F + I_G. \end{split}
\]
The upper bound in Theorem 1 follows from the following two lemmas.
\begin{lemma} [Br\"udern]   We have 
\[  I_F \ll \sqrt{x}. \]
\end{lemma}

\begin{lemma}   We have 
\[  I_G \ll \sqrt{x}\log x . \]
\end{lemma}

In what follows we always assume $(a,q)=1$, and  define the new variable $\beta$ by
\begin{equation} \label{alpha} \alpha = \frac{a}{q} +\beta .  \end{equation}
\begin{proof}[Proof of Lemma 2]  The proof follows from the estimate 
\begin{equation} \label{Festimate}
F_q(\alpha)
 \ll \left\{ \begin{array}{ll}
      {\min\left(x,\frac{1}{|\beta|}\right)\frac{\log
\frac{2\sqrt{x}}{q}}{q} ,} & \mbox{if $q\le \sqrt{x}$, $ |\beta|\le \frac{1}{2q\sqrt{x}}$;} \\
      0, & \mbox{if $q>\sqrt{x}$;} \\
\end{array}
\right.
\end{equation}
since this implies
\[   \begin{split}   I_F &\ll \sum_{q\le \sqrt{x}} q \int_0^{1/(2q\sqrt{x})}\min\left(x, \frac{1}{|\beta|}\right) \frac{\log
\frac{2\sqrt{x}}{q}}{q}\, d\beta \\&
\ll \sum_{q\le \sqrt{x}} \log
\frac{2\sqrt{x}}{q}\left(\int_0^{1/(2x)} x \, d\beta +  \int_{1/(2x)}^{1/(2q\sqrt{x})} \frac{1}{\beta} \, d\beta \right) \\& \ll \sum_{q\le \sqrt{x}} \left(
\log
\frac{2\sqrt{x}}{q}\right)^2 \quad \ll \sqrt{x}.
\end{split} \]
To prove \eqref{Festimate}, we first note that the conditions $q|u$ and $u \le \sqrt{x}$ force $F_q(\alpha)=0$ when $q>\sqrt{x}$. Next, when $q\le \sqrt{x}$ we write $u=jq$ and have
\[ F_q(\alpha) = \sum_{j\le \frac{\sqrt{x}}{q}}\sum_{jq\le v\le \frac{x}{jq}} e( jqv\beta).\]
Making use of the estimate
\begin{equation} \label{basic} \sum_{N_1 < n \le N_2} e(n\alpha)  \ll  \min\left( N_2-N_1, \frac{1}{\Vert \alpha \Vert}\right)\end{equation}
we have 
\[ F_q(\alpha) \ll \sum_{j\le \frac{\sqrt{x}}{q}}\min\left(\frac{x}{jq}, \frac{1}{\Vert jq\beta \Vert}\right).\]
In this sum $jq\le \sqrt{x}$ so that $|jq\beta | \le |\beta| \sqrt{x}$, and hence the condition $ |\beta|\le \frac{1}{2q\sqrt{x}}$ implies $|jq\beta | \le \frac{1}{2q} \le \frac12$. Hence $ \Vert jq\beta \Vert = jq|\beta|$, and we have
\[ F_q(\alpha) \ll \sum_{j\le \frac{\sqrt{x}}{q} } \frac{1}{jq} \min\left(x, \frac{1}{|\beta|}\right)\ll \min(x, \frac{1}{|\beta|})\frac{\log\frac{2\sqrt{x}}{q}}{q}  .\] 
\end{proof}

\begin{proof}[Proof of Lemma 3] The proof follows from the estimate, 
\begin{equation} \label{Gestimate}
G_q(\alpha) \ll (\sqrt{x} + q)\log q , \quad \text{for} \  \alpha = \frac{a}{q} + \beta, \quad  |\beta| \le \frac{1}{2q\sqrt{x}}, \end{equation}
since this implies
\[   \begin{split}   I_G &\ll \sum_{q\le Q} q \int_0^{1/(2q\sqrt{x})}(\sqrt{x} +q) \log q\,  d\beta \\&
\ll  \frac{1}{\sqrt{x}}( Q(\sqrt{x} +Q)\log Q \ll \sqrt{x}\log x
\end{split} \]
by \eqref{Qequation}.
To prove \eqref{Gestimate}, we apply \eqref{basic} to the sum over $v$ in $G_q(\alpha)$ and obtain
\[  G_q(\alpha) \ll  \sum_{\substack{u\le\sqrt{x}\\ q\nmid u}} \min\left( \frac{x}{u},\frac{1}{\Vert u\alpha \Vert}\right).  \]
Recalling $\Vert x \Vert =  \Vert -x \Vert $ and the triangle inequality $\Vert  x+y\Vert \le \Vert x\Vert +\Vert y\Vert$, and using the conditions $1\le u\le \sqrt{x}$, $q\nmid u$, $ |\beta| \le \frac{1}{2q\sqrt{x}}$, we have
\[  \begin{split}\left \Vert u\alpha\right \Vert &\ge \left\Vert \frac{au}{q}\right \Vert - \Vert u \beta \Vert \\&
\ge \left \Vert \frac{au}{q} \right\Vert - u|\beta| \\&
\ge\left \Vert \frac{au}{q} \right\Vert -  \frac{u} {2q\sqrt{x}} \\&
\ge \left \Vert \frac{au}{q} \right \Vert -  \frac{1} {2q}   \\&
\ge \frac12\left \Vert \frac{au}{q} \right \Vert , \end{split} \]
and therefore
 \[  G_q(\alpha) \ll  \sum_{\substack{u\le\sqrt{x}\\ q\nmid u}} \frac{1}{\left\Vert \frac{au}{q} \right \Vert}.  \]
Here $\left\Vert \frac{au}{q} \right \Vert = \frac{b}{q}$ for some integer $1\le b \le \frac{q}{2}$ and since the  integers $\{au: 1\le u \le \frac{q}{2}\}$ are distinct modulo $q$ since $(a,q)=1$, we see
\[ \sum_{1\le u \le \frac{q}{2} }  \frac{1}{\left\Vert \frac{au}{q}\right \Vert} = \sum_{1\le b \le \frac{q}{2}}\frac{q}{b} \ll  q\log q .\]
If  $q> \sqrt{x}$  then 
\[  \sum_{\substack{u\le\sqrt{x}\\ q\nmid u}} \frac{1}{\left\Vert \frac{au}{q} \right \Vert} \le 2   \sum_{1\le u \le \frac{q}{2} }  \frac{1}{\left\Vert \frac{au}{q}\right \Vert} \ll q\log q ,\]
while if $q\le \sqrt{x}$ then
 the sum bounding $G_q(\alpha)$ can be split into  $\ll \frac{ \sqrt{x}}{q}$ sums of this type and 
 \[ G_q(\alpha) \ll  \frac{\sqrt{x}}{q} (q\log q )\ll \sqrt{x}\log q. \]
\end{proof}

\section{Proof of the lower bound}
Following Br\"udern, consider the intervals  $|\alpha - \frac{a}{q}| \le 1/(4x)$ for $1\le a\le q\le Q $, where we take $\frac12 \sqrt{x} \le Q \le \sqrt{x}$. These intervals are pairwise disjoint because for two distinct fractions
$|a/q - a'/q'| \ge 1/(qq') \ge 1/x$. (We will see later why these intervals have been chosen shorter than required to be disjoint.) Hence, using \eqref{alpha}
\[ \int_0^1 |S(\alpha)|\, d\alpha = \int_{1/Q}^{1+1/Q} |S(\alpha)|\, d\alpha  \ge \sum_{q\le \frac12\sqrt{x} }\sum_{\substack{a = 1\\ (a,q)=1}}^q \int_{-1/(4x)}^{1/(4x)} \left\vert S\left(\frac{a}{q} +\beta\right)\right\vert d\beta. 
\]
Next we follow Vaughan's method  \cite{Va} and apply the triangle inequality to obtain the lower bound
\[ \int_0^1 |S(\alpha)|\, d\alpha  \ge \sum_{q\le \frac12\sqrt{x} } \int_{-1/(4x)}^{1/(4x)} \left\vert \sum_{\substack{a = 1\\ (a,q)=1}}^qS\left(\frac{a}{q} +\beta\right)\right\vert d\beta. 
\]
Letting
\begin{equation} \label{Uq}  U_q(x;\beta) := \sum_{\substack{a = 1\\ (a,q)=1}}^qS\left(\frac{a}{q} +\beta\right) = \sum_{n\le x} \tau(n) c_q(n)e(n\beta) , \end{equation}
where 
\[ c_q(n) = \sum_{\substack{a = 1\\ (a,q)=1}}^q e\left(\frac{an}{q}\right) \]
is the Ramanujan sum,  our lower bound may now be written as
\begin{equation} \label{lowerbound} \int_0^1 |S(\alpha)|\, d\alpha  \ge \sum_{q\le \frac12\sqrt{x} } \int_{-1/(4x)}^{1/(4x)} \left\vert U_q(x;\beta)\right\vert d\beta. \end{equation}
To complete the proof of the lower bound we need the following lemma, which we prove at the end of this section. 
\begin{lemma}
\label{Usum}
For $	q\ge 1$ we have 
\begin{align}
U_q(x; 0) = \frac{\varphi(q)}{q}x(\log (x/q^2) + 2\gamma - 1) + O(q\tau(q) (x^{\frac13} + q^{\frac12})x^\epsilon),
\end{align}
where $\gamma$ is Euler's constant. 
\end{lemma}

\begin{proof}[Proof of the lower bound in Theorem 1]  For any exponential sum
$T(x; \beta) =\sum_{n\le x} a_n e(n\beta)$ we have by partial summation or direct verification
\[ T(x; \beta) =  e(\beta x) T(x;0) - 2\pi i \beta \int_1^x e(\beta y)  T(y;0) \, dy.\] 
Taking $T(x;\beta) = U_q(x;\beta)$ we thus obtain from \eqref{lowerbound} and the triangle inequality
\begin{equation} \label{lower} \int_0^1 |S(\alpha)|\, d\alpha  \ge \sum_{q\le \frac12\sqrt{x} } \int_{-1/(4x)}^{1/(4x)} \left( |U_q(x;0)| - 2\pi |\beta | \int_1^x |U_q(y; 0)| \, dy \right) d\beta.\end{equation}
By Lemma 4,  with $q\le \frac12\sqrt{x}$, 
\[ \begin{split}    \int_1^x |U_q(y; 0)| \, dy & \le \frac{\varphi(q)}{q} \left( \int_1^x  y | \log(y/q^2)| +(2\gamma -1)y \, dy\right)  +   O(xq\tau(q) (x^{\frac13} + q^{\frac12})x^\epsilon) \\&
\le  \frac{\varphi(q)}{q} \left( \int_1^{q^2}  y \log(q^2/y) \, dy + \int_{q^2}^x  y \log(y/q^2) \, dy+ \frac12 x^2 (2\gamma -1) \right) \\ & \hskip 2.5in +   O(xq\tau(q) (x^{\frac13} + q^{\frac12})x^\epsilon) \\&
= \frac{x}{2}\left( \frac{\varphi(q)}{q}\left( x(\log (x/q^2)+ 2\gamma -1) - \frac{x}{2}  + \frac{q^4}{x} \right) +   O(q\tau(q) (x^{\frac13} + q^{\frac12})x^\epsilon)\right)\\&
\le \frac{x}{2} U_q(x,0)  +   O(q\tau(q) (x^{\frac13} + q^{\frac12})x^\epsilon).
\end{split}\]
Using $|\beta | \le 1/(4x)$, we have
\[  |U_q(x;0)| - 2\pi \beta \int_1^x |U_q(y; 0)| \, dy  \ge \left( 1 - \frac{\pi}{4} \right)|U_q(x;0)|  -  O(q\tau(q) (x^{\frac13} + q^{\frac12})x^\epsilon). \]
We conclude, returning to \eqref{lower} and making use of Lemma 4 again, 
\[ \begin{split}  \int_0^1 |S(\alpha)|\, d\alpha & \ge \frac{4-\pi}{8 x} \sum_{q\le \frac12\sqrt{x} } \left(|U_q(x;0)| - O(q\tau(q) (x^{\frac13} + q^{\frac12})x^\epsilon)\right) \\&
\ge \frac{4-\pi}{8} \sum_{q\le \frac12\sqrt{x} } \frac{\varphi(q)}{q}(\log (x/q^2) + 2\gamma - 1) -O(x^{\frac13+\epsilon}) .\end{split} \]
It is easy to see that the sum above is $\gg  \sqrt{x} $ which suffices to proves the lower bound.  More precisely, using 
\[ \frac{\varphi(n)}{n} = \sum_{d|n} \frac{ \mu(d)  }{d} \]
a simple calculation gives the well-known result
\[ \sum_{ n\le x} \frac{\varphi(n)}{n} = \frac{6}{\pi^2} x +O(\log x), \]
and then by partial summation we find 
\[ \sum_{q\le \frac12\sqrt{x} } \frac{\varphi(q)}{q}(\log (x/q^2) + 2\gamma - 1)\sim \frac{6}{\pi^2}(\log 2 +\gamma -1)\sqrt{x} .\]
\end{proof}
\begin{proof}[Proof of Lemma 4]  Pongsriiam and Vaughan \cite{PV} recently proved the following very useful result on the divisor 
function in arithmetic progressions.  For inteqer $a$ and $d\ge 1$ and real $x\ge 1$ we have
\[  \sum_{\substack{ n\le x \\ n\equiv a(\text{mod}\, d)}}\tau(n) = \frac{x}{d} \sum_{r | d}\frac{c_r(a)}{r}\left(\log\frac{x}{r^2} + 2\gamma -1\right) +O( (x^{\frac13} + d^{\frac12})x^\epsilon),\]
where $\gamma$ is Euler's constant and $c_r(a)$ is the Ramanujan sum. We need the special case when $a=0$ which along with the situation $(a,d)>1$ is explicitly allowed in this formula.  Hence we have 
\begin{equation} \label{PV}
  \sum_{\substack{ n\le x \\ d|n}} \tau(n) = \frac{x}{d}f_x(d) + O( (x^{\frac13} + d^{\frac12})x^\epsilon)
\end{equation}
where 
\begin{equation}\label{f-g}
f_x(d) = \sum_{r|d} g_x(r),    \quad   g_x(r) = \frac{\varphi(r)}{r}(\log(x/r^2) + 2\gamma - 1).  
\end{equation}
Making use of 
\[
c_q(n) = \sum_{d|(n, q)}d\mu\left(\frac{q}{d}\right),
\]
and  \eqref{PV} we have 
\[ \begin{split} U_q(x;0) &= \sum_{n\le x}\tau(n)c_q(n)\\&
= \sum_{d|q}d\mu\left(\frac{q}{d}\right)\us{d|n}{\sum_{n\le x}}\tau(n) \\& = 
x\sum_{d|q}\mu\left(\frac{q}{d}\right)f_x(d) +  O( q\tau(q)(x^{\frac13} + d^{\frac12})x^\epsilon).
 \end{split} \]
We evaluate the sum above using  Dirichlet convolution and the identity $1*\mu  = \delta$ where $\delta(n)$ is the identity for Dirichlet convolution defined to be 1 if $n=1$ and zero otherwise.  Hence  
\[ \sum_{d|q}\mu\left(\frac{q}{d}\right)f_x(d) =  (f_x * \mu)(q)   = ( (g_x*1)*\mu)(q) =   (g_x * \delta)(q)  = g_x(q),
\]
and Lemma 4 is proved. 
\end{proof}

\vspace{1cm} \footnotesize D. A. Goldston  \,\,\,
(daniel.goldston@sjsu.edu)

Department of Mathematics and Statistics

San Jose State University

San Jose, CA 95192

 USA \\

\vspace{1cm} \footnotesize M. Pandey  \,\,\,
(mayankpandey9973@gmail.com)

12861 Regan Lane

Saratoga, CA 95070

 USA \\

\end{document}